\documentclass[11pt]{article}
\usepackage{amssymb}
\usepackage{amsmath}
\usepackage{amsthm}
\usepackage{geometry} 
\usepackage{hyperref}
\usepackage{parskip}
\usepackage{comment}
\usepackage{enumerate}
\usepackage[toc]{appendix}
\usepackage[scr=rsfso,frak=euler]{mathalfa}

\usepackage{soul,xcolor}

\usepackage{chemformula}
 
\usepackage{lineno}

\usepackage{enumerate}

\usepackage{bbm}
\newcommand{\ind}[1]{\mathbbm{1}_{\{#1\}}}

\RequirePackage{geometry}
\newtheorem{theorem}{Theorem}[section]
\newtheorem{proposition}[theorem]{Proposition}
\newtheorem{lemma}[theorem]{Lemma}

\theoremstyle{remark}
\newtheorem{remark}{Remark}
\newtheorem{definition}[theorem]{Definition}

\definecolor{darkgreen}{rgb}{0,0.7,0}
\definecolor{darkred}{rgb}{0.5,0,0}
\definecolor{ultramarine}{rgb}{0.07, 0.04, 0.56}

\usepackage[
    backend=biber,
    style=numeric,
    sorting=nyt,
    maxbibnames=10,
  ]{biblatex}
\usepackage{authblk}
\begin{document}


\title{Killed path-dependent McKean–Vlasov SDEs for a\\probabilistic representation of non-conservative McKean PDEs}
 
\author{Daniela Morale\thanks{Dept. of Mathematics, University of Milano, \href{mailto:daniela.morale@unimi.it}{daniela.morale@unimi.it}} , Leonardo Tarquini\thanks{Dept. of Mathematics, University of Oslo, \href{mailto:leonardt@math.uio.no}{leonardt@math.uio.no}} , Stefania Ugolini\thanks{Dept. of Mathematics, University of Milano, \href{mailto:stefania.ugolini@unimi.it}{stefania.ugolini@unimi.it}}}

\maketitle
\begin{abstract} 
    A McKean-Vlasov stochastic differential equation subject to killing associated to a  regularised non-conservative and path-dependent nonlinear parabolic partial differential equation is studied. The existence and pathwise  uniqueness of a strong solution and the regularity properties of its sub-probability law are proved. The density of such a law may be seen as a weak solution of the considered PDE. The well-posedness  of the associated particle system is also discussed.
\end{abstract}



\section{Introduction}\label{sec:introduction}
In this paper, we derive a McKean-Vlasov-type stochastic differential equation (SDE) with killing, whose time marginals satisfy a  highly nonlinear, non-conservative, and path-dependent McKean-type partial differential equation (PDE). The main goal is to conduct an analytical study of the well-posedness of both such an SDE and its associated particle system.

The reaction-diffusion PDE model reads as follows. For any  $(t, x) \in (0, T] \times \mathbb{R}$,  
\begin{equation}\label{eq:PDE_rho_regularised}
    \begin{split}
        \partial_t\rho(t,x) \,=\, &\Delta\rho (t,x)\,-\, \nabla \cdot \left[ b\big(K*\rho (\cdot,x)(t), \nabla K*\rho(\cdot,x)(t)\big) \rho(t,x)\, \right] \\
        &-c\left(K*\rho(\cdot,x)(t)\right)\rho(t,x),
    \end{split}
\end{equation}
with initial condition given by the probability density \( \rho(0,x) = \rho_0(x)\), where the advection term $b$ and the reaction coefficient $c$ satisfy suitable regularity properties.

In \eqref{eq:PDE_rho_regularised}, \(K\) is a smooth mollifier (to be defined later), and \(*\) denotes the spatial convolution operation. The notation \(\rho(\cdot, x)(t)\) stands for the time integral of \(\rho\) at position \(x\), that is,
\begin{equation}\label{eq:def_rho_integral_and_convolution}
\begin{split}
\rho(\cdot,x)(t) &:= \int_0^t \rho(s,x)\,ds, \\
K * \rho(\cdot,x)(t) &:= \int_0^t \int_{\mathbb{R}} K(x - y)\rho(y,s)\,dy\,ds.
\end{split}
\end{equation}
 
Equation \eqref{eq:PDE_rho_regularised} is non-conservative due to the reaction term. Therefore, whenever $\rho_0$ is a probability density, for every $t\in [0,T]$,
 \begin{equation}\label{eq:rho_prob_density}
     \int_{\mathbb R} \rho(t,x)dx\leq 1.
 \end{equation}

We propose a probabilistic interpretation of equation \eqref{eq:PDE_rho_regularised}, viewed as the macroscopic model describing the evolution of the density of a sub-probability measure \(\nu_t\). This measure is defined on \(\mathcal{B}_{\mathbb{R}}\) by
\[
\nu_t = \mathbb{P}\left( X_t \in \cdot,\ t < \tau \right),
\]
where \(X = \{X_t\}_t\) is a stochastic process solving the following killed McKean--Vlasov SDE
\[
X_t = X_0 + \int_0^t b\left(\int_0^s K*\nu_r(X_s)\,dr,\ \int_0^s \nabla K*\nu_r(X_s)\,dr\right)\,ds + \sqrt{2} W_t; \quad t < \tau,
\]
where  the stopping time $\tau$, defined as
\begin{equation*}
    \tau \,:=\, \inf_{t\geq 0}\Bigg\{ \int_0^t c\left( \int_0^s K*\nu_r(X_s)dr \right)ds \,\geq\, Z \Bigg\},
\end{equation*}
represents the reaction time of the typical particle that moves at the microscale following the SDE.

Motivated by applications, in the present paper, we consider a specific form of the coefficients $b$ and $c$, for which we prove that the aforementioned regularity conditions are satisfied. Specifically, the advection     and the reaction  terms are  defined, respectively as   
\begin{equation}\label{eq:def_b_PDE}
\begin{aligned}
    b(x, y) &:= -\varphi_1\lambda c_0 \frac{e^{-\lambda x} y}{\varphi_0 + \varphi_1 c_0 e^{-\lambda x}},  &&\quad  (x, y) \in \mathbb{R}^2, \\
    c(x) &:= \lambda c_0 e^{-\lambda x}, &&\quad   x \in \mathbb{R}
\end{aligned}
\end{equation}
where $\lambda \in \mathbb{R}_+.$
The parameters $\varphi_0\in\mathbb{R}_+$ and $\varphi_1\in\mathbb{R}$ are such that there exist two positive constants $m$ and $M$ satisfying, for any $(t,x) \in [0,T]\times\mathbb{R}$,
\begin{equation*}
    m\leq \varphi_0+\varphi_1c_0 \exp\big( -\lambda \rho(\cdot,x)(t)\big)\leq M.
\end{equation*}

A first probabilistic representation of the PDE system \eqref{eq:PDE_rho_regularised} with $b$ and $c$ given by \eqref{eq:def_b_PDE}, respectively, has been proposed by the same authors in \cite{2024_MTU_arxiv} as the mean-field limit of a system of Brownian particles governed by a non-conservative evolution. Specifically, the microscopic model consists of a nonlinear SDE coupled with a Feynman–Kac-type equation, whose well-posedness is established along with its connection to \eqref{eq:PDE_rho_regularised},\eqref{eq:def_b_PDE}. Propagation of chaos for the associated particle system to such a stochastic model is also proved, showing convergence of the empirical measure to the solution of the PDE. The aforementioned work extends the probabilistic approach to non-conservative McKean-type equations proposed in \cite{2016_Russo,2021_IzydorczykOudjaneRusso} to the path-dependent setting.

In the present paper, we adopt a complementary modelling perspective based on a killed McKean-Vlasov SDE, where the killing rate is associated to the reaction term of the PDE model \eqref{eq:PDE_rho_regularised},\eqref{eq:def_b_PDE}. We prove well-posedness and show that the time marginals of the process yield a weak solution to the PDE. The novelty of the present work lies in introducing an equivalent formulation from a more probabilistic perspective, based on a killed process, and in clarifying its connection to the previous analytic representation \cite{2024_MTU_arxiv}. Moreover, the present paper provides a unified treatment of McKean-Vlasov SDEs with killing and their associated non-conservative PDEs, particularly in the path-dependent case. Also, it can be easily extended to the case of a generic diffusion coefficient.

Beyond its analytical and probabilistic viewpoints, the study is also motivated by its potential real world applications. In fact, equation \eqref{eq:PDE_rho_regularised} is a regularised version of  a specific PDE coupled with an ordinary differential equation (ODE), known within the cultural heritage literature of marble degradation, in particular for the modelling of the reaction of calcium carbonate  under the interaction with diffusing sulphur dioxide  \cite{2007_AFNT2_TPM, 2005_GN_NLA}. The regularisation consists in replacing the spatially local dependence of the advection and reaction terms on \(\rho\) with a spatially non-local one, mediated by the convolution with a smooth kernel \(K\). More precisely, the advection and reaction terms now depend on the measure \(\rho(x,t)\,dx\) through the convolution \(K * \rho\), which can be interpreted as a non-local, or \emph{regularised}, concentration. For a detailed derivation of the mentioned ODE-PDE system, of the regularised version, and for an account of the qualitative behaviour of the solution on the half-line, interested readers may consult \cite{2004_ADN, 2025_MauMorUgo,2025_MACH2023_AMMU,2024_MTU_arxiv} and the references therein.
 
Over the past few decades, many reaction-diffusion models for the sulphation phenomenon have emerged within the deterministic framework and have been extensively investigated both analytically and numerically (see, e.g.,\cite{2007_AFNT_TPM,2007_AFNT2_TPM,2023_Bonetti_natalini_NLA,2019_BCFGN_CPAA,2005_GN_NLA,2007_GN_CPDE}). However, some stochastic approaches to degradation modelling have only recently emerged \cite{MACH2021_AGMMU,2024_ArceciMoraleUgolini_arxiv,2025_MauMorUgo,2025_Mach2023_MRU,2024_MTU_arxiv}. At the macroscale, an extension of the dynamics described by \eqref{eq:PDE_rho_regularised} has been proposed by coupling it with a source of randomness through a stochastic dynamical boundary condition, and  the well posedness of such  a random PDE has been studied \cite{2025_MACH2023_AMMU,2025_MauMorUgo}. The qualitative behaviour of the solutions to such a stochastic boundary value problem has been investigated from a numerical point of view in \cite{2024_ArceciMoraleUgolini_arxiv}, and some rigorous convergence results of some related discrete schemes in a discrete Besov space framework have been established in \cite{2025_ADMU_arxiv}. Furthermore, two stochastic interacting particle models have been proposed at a nanoscale \cite{2025_Mach2023_MRU,2025_JMMRU_Arxiv}. While a complete analytical investigation is still lacking, in the cited papers, stochastic particle models for the dynamics of the calcium carbonate, sulphur dioxide, and gypsum molecules have been introduced and the qualitative behaviour of their solutions  have been investigated. The work presented in \cite{2024_MTU_arxiv}, together with the current paper, aims at providing a rigorous probabilistic foundation for the PDE system \eqref{eq:PDE_rho_regularised},\eqref{eq:def_b_PDE} by identifying and analysing the underlying microscopic stochastic dynamics. These studies address to interpret the PDE as a macroscopic mean-field limit of a system of interacting Brownian particles, offering a microscale model able to capture essential features of the physical phenomenon and, particularly, of the reactive and non-conservative nature of the dynamics.

The paper is organized as follows. In Section \ref{sec:killed_SDE}, the killed McKean-Vlasov SDE associated with the PDE model \eqref{eq:PDE_rho_regularised},\eqref{eq:def_b_PDE} is introduced. Specifically, after proving some regularity properties of the coefficients   \eqref{eq:def_b_PDE}, the well-posedness of the stochastic model is established. In Section \ref{sec:probabilistic_interpretation}, the link between the killed McKean-Vlasov SDE and the PDE model \eqref{eq:PDE_rho_regularised},\eqref{eq:def_b_PDE} is described. Finally, in Section \ref{sec:PS_and_chaos}, the particle system associated with the stochastic model is introduced and its well-posedness is proved.

\paragraph{Probability space and cemetery state.} Let $(\Omega, \mathcal{F},\mathbb{F}=\{\mathcal{F}_t\}_{t\in [0,T]},\mathbb{P})$ be a filtrated probability space and choose $T \in \mathbb R_+$ to be a sufficiently large time horizon. All stochastic processes throughout the paper are assumed to be defined and adapted to $\mathbb{F}$.  

We denote by $\{\Delta\}$ a \emph{cemetery state} isolated from $\mathbb{R}$ and consider the convention that any function $f$ defined on $\mathbb{R}$ can be extended to $\mathbb{R}\cup\{\Delta\}$ by setting $f(\Delta)=0$.

\section{A killed McKean-Vlasov SDE}\label{sec:killed_SDE}
Let $\tau$ be an $\mathbb{F}$-stopping time and $I\left(\tau,T\right):=[0,\tau)\cup[0, T]$. Let $X=\{X_t\}_{t\in[0,T]}$ be an $\mathbb R \cup \{\Delta\}$ - valued stochastic process solution of the following killed McKean-Vlasov stochastic differential equation
\begin{equation}\label{eq:killed_SDE_measure}
    \begin{aligned}
        X_t &= X_0 + \int_0^t b\left(\int_0^s K*\nu_r(X_s)dr, \int_0^s\nabla K*\nu_r(X_s)dr\right)ds + \sqrt{2}W_t,
        &\quad t &\in I\left(\tau,T\right); \\
        X_t &\in \{\Delta\}, 
        &\quad t &\geq \tau,
    \end{aligned}
\end{equation}
where for any  $t\in[0,T]$, $\nu_t$ is the law of $X_t$ whenever $t$ is a survival time, i.e. it is the sub-probability measure  defined on $\mathcal{B}_{\mathbb R}$ as $\nu_t(\cdot) =\mathbb{P}\left( X_t\in \cdot, t<\tau \right)$.

System \eqref{eq:killed_SDE_measure} is subject to the initial conditions
\begin{equation}\label{eq:inital_datum_SDE}
    X_0 \sim \zeta_0;\quad \mathbb{E}[|\zeta_0|^2]< \infty;\quad \mathcal{L}(\zeta_0)=\nu_0;\quad \nu_0(dx)=\rho_0(x)dx;\quad \rho_0 \in L^2(\mathbb{R})\cap C_b(\mathbb R).
\end{equation}
For example, one could take $\zeta_0$ to have normal distribution. Moreover, it is adopted as convention $X_{\tau}:=\lim_{t\rightarrow \tau^{-}}X_t$.

The stopping time $\tau=\tau(X)$ represents the reaction time and it is defined as
\begin{equation}\label{eq:def_killing_measure}
    \tau =\tau(X) \,:=\, \inf_{t\geq 0}\Bigg\{ \int_0^t \lambda c_0\exp\left( -\lambda \int_0^s K*\nu_r(X_s)dr \right)ds \,\geq\, Z \Bigg\},
\end{equation}
where $Z$ is an exponentially distributed random variable with mean 1, ${Z\sim Exp(1)}$, independent of $X$.

From the above definition of $\tau$, we derive the survival probability
\begin{equation*}
    \mathbb{P}\left(\tau>t|\mathcal{F}_t\right) = \mathbb E\left[ \mathbbm{1}_{(t,\infty)}\left(\tau(X)\right) |\mathcal{F}_t\right]= \exp\left(-\int_0^t \lambda c_0\exp\left( -\lambda \int_0^s K*\nu_r(X_s)dr \right)ds \right).
\end{equation*}
This yields the following equation for  $\nu_t$. For any $A\in \mathcal{B}_{\mathbb R}$,
\begin{equation}\label{eq:subprob_killed_SDE}
    \begin{aligned}
        \nu_t(A)&=\mathbb{E}\left[\mathbbm{1}_{A}(X_t)\mathbbm{1}_{(t,\infty)}\left(\tau(X)\right)\right]\\
        &=\mathbb{E}\left[\mathbbm{1}_{A}(X_t)\mathbb{E}\left[\mathbbm{1}_{(t,\infty)}\left( \tau(X)\right)|\mathcal{F}_t\right]\right]\\
        &=\mathbb{E}\left[\mathbbm{1}_{A}(X_t) \exp\left(-\int_0^t \lambda c_0\exp\left( -\lambda \int_0^s K*\nu_r(X_s)dr \right)ds \right)\right].
    \end{aligned}
\end{equation}
Thus, by \eqref{eq:subprob_killed_SDE}, for any test function $f\in C_b^{\infty}(\mathbb{R})$,
\begin{equation}\label{eq:subprob_killed_SDE_int}
    \int_{\mathbb{R}}f(x)\nu_t(dx)=\mathbb{E}\left[f(X_t) \exp\left(-\int_0^t \lambda c_0\exp\left( -\lambda \int_0^s K*\nu_r(X_s)dr \right)ds \right)\right].
\end{equation}
\begin{remark}\label{remark:linking_equation}
    Equation \eqref{eq:subprob_killed_SDE_int} reads as the weak version of the Feynman-Kac-type equation considered in \cite{2024_MTU_arxiv}.
\end{remark}

We may extend $\nu_t$ to the full space $\mathbb{R} \cup \{\Delta\}$ by setting
\begin{equation*}
    \nu_t(\{\Delta\})=\mathbb{P}\left( X_t\in \{\Delta\}, t\geq \tau \right)=1-\nu_t(\mathbb R),
\end{equation*}
reflecting the total mass lost due to the killing.

The stopping time $\tau$ may alternatively be viewed as the first and only jump of an inhomogeneous Poisson point process. Let $N^0$ be a standard Poisson process independent of $W$ and $X_0$, and define the time and space-dependent intensity function
\begin{equation*}
    \overline{\lambda}\left(t,x\right):= \lambda c_0\exp\left( -\lambda\int_0^t K*\nu_s(x)ds\right)\quad (t,x)\in[0,T]\times\mathbb{R}.
\end{equation*}
Also defined is the cumulative intensity process
\begin{equation}\label{eq:def_compensator_poisson}
    \Lambda_t := \int_0^t \bar{\lambda}\left(s,X_t\right)ds
\end{equation}
and constructed is an inhomogeneous Poisson process via the time change
\begin{equation}\label{eq:def_poisson}
    N_t := N_{\Lambda_t}^{0}.
\end{equation}
Importantly, the process $\Lambda_t$ serves as the compensator of $N$. Then, the killing time $\tau$ can be equivalently expressed as the first jump time of $N$, i.e.,
\begin{equation*}
    \tau:=\inf\{t\geq 0 \,:\, N_t=1\}.
\end{equation*}
This alternative representation allows for another expression of $\nu_t$. That is, for any $A\in\mathcal{B}_{\mathbb{R}}$,
\begin{equation}\label{eq:subprob_killed_SDE_2}
    \nu_t(A)= \mathbb{E}\left[\mathbbm{1}_{A}(X_t) \exp\left(- \Lambda_t \right)\right]=\mathbb P\left( X_t\in A, \Lambda_t<Z\right).
\end{equation}

\subsection{Regularity properties of the PDE coefficients}
Let us first establish the kind of kernel $K$ we consider for modelling the non local interaction.
\begin{definition}\label{def:kernel}
     A kernel $K: \mathbb{R} \rightarrow \mathbb{R}^+$ is said to be a \emph{smooth mollifier} if it satisfies the following properties:
\begin{itemize}
    \item[i)]  $K\in C^{\infty}(\mathbb{R})$.
    \item[ii)]  $K$ is bounded with its first and second derivate, i.e.  there exist constants $M_K,M_K^\prime,M_K^{\prime\prime}\in \mathbb R_+$ such that for any $y\in\mathbb{R}$,
    \begin{equation}\label{eq:K_MK}
     |K(y)|\leq M_K, \quad  |\nabla K(y)|\leq M_K^\prime, \quad  |\nabla^2 K(y)|\leq M_K^{\prime\prime}.   
    \end{equation}
    \item[iii)] $K$ and its gradient $\nabla K$ are Lipschitz continuous, i.e. there exist two constants $L_K,L'_K \in \mathbb R_+$  such that for any $y,y'\in\mathbb{R}$,
    \begin{align}
        |K(y')\,-\,K(y)|\,&\leq\,L_K|y'\,-\,y|;\label{eq:K_LK}&&\\
        |\nabla K(y')\,-\,\nabla K(y)|\,&\leq\,L'_K|y'\,-\,y|.\label{eq:nablaK_LprimeK}&&
    \end{align}
    \item[iv)] $K$ is a probability density, that is $\int_{\mathbb{R}}K(y)\,dy \,=\,1$.
\end{itemize}
\end{definition}
\begin{remark}
    An example of smooth mollifier that satisfies the assumptions of Definition \ref{def:kernel} is the Gaussian kernel with standard deviation $\sigma\in \mathbb R_+$,
    \begin{equation*}
        G(x,\sigma):=\frac{1}{\sqrt{2\pi\sigma^2}}\exp\left(-\frac{x^2}{2\sigma^2}\right).
    \end{equation*}
\end{remark}

Before proving the well-posedness of the system \eqref{eq:killed_SDE_measure},\eqref{eq:def_killing_measure}, we first outline some regularity properties for the advection and reaction terms \eqref{eq:def_b_PDE}.

\begin{proposition}\label{prop:continuity_boundedness_b}
   Let $D$ be a compact subset of $\mathbb{R}^+_0$.  For any $x,x',y,y'\in \mathbb{R}^+_0$ with $y,y'\in D$,   the function  $b$   in \eqref{eq:def_b_PDE} satisfies the following continuity property
    \begin{equation*}
        \Big|b(x,y) - b(x',y')\Big|\leq C\Big( |x-x'|+ |y-y'| \Big),
    \end{equation*}
    with $C=C\left(\varphi_0,\varphi_1,\lambda,c_0,M_K'\right)$. Moreover, $b$ is bounded, i.e.  there exists a constant $M_b\in\mathbb{R}_{+}$ such that
    \begin{equation}\label{eq:bound_drift}
        \Big|b(x,y)\Big|\leq M_b.
    \end{equation}
\end{proposition}
\begin{proof}
   See \cite{2024_MTU_arxiv}, Proposition 3.1 and Proposition 3.2.
\end{proof}
 
\begin{proposition} \label{prop:lipschitz_prop_b}
    Let $\rho(t,x)$ be the solution of \eqref{eq:PDE_rho_regularised},\eqref{eq:def_b_PDE} with initial condition satisfying \eqref{eq:rho_prob_density}. The drift $b$ in \eqref{eq:def_b_PDE} is Lipschitz continuous. Precisely, for any $z,z'\in\mathbb{R}\cup\{\Delta\}$ and $t\in[0,T]$,
    \begin{equation*}
        \begin{aligned}
            &\Big| b \Big(K*\rho(\cdot,x)(t), \nabla K*\rho(\cdot,x)(t)\Big) - b \Big(K*\rho(\cdot,y)(t), \nabla K*\rho(\cdot,y)(t)\Big) \Big|\\
            &\leq 2C\left(L_K\vee L_K^\prime\right)T |x - y|.
        \end{aligned}
    \end{equation*}
\end{proposition}
\begin{proof}
    Due to \eqref{eq:K_MK} and \eqref{eq:rho_prob_density}, and since the system \eqref{eq:PDE_rho_regularised},\eqref{eq:def_b_PDE} is not conservative, we get the following  bounds for the convoluted density \eqref{eq:def_rho_integral_and_convolution}. For any $(t,x) \in [0,T]\times \mathbb R$,
    \begin{equation}\label{eq:boundness_convolution}
        \begin{aligned}
            &0\leq K*\rho(\cdot,x)(t) \leq M_K T,\quad\\
            &0\leq \nabla K*\rho(\cdot,x)(t) \leq M_K^\prime T.
        \end{aligned}
    \end{equation}
    
    Therefore, by \eqref{eq:K_LK}, \eqref{eq:nablaK_LprimeK}, and Proposition \ref{prop:continuity_boundedness_b}, 
    \begin{align*}
        &\Big| b \Big(K*\rho(\cdot,x)(t), \nabla K*\rho(\cdot,x)(t) \Big) - b \Big(K*\rho(\cdot,y)(t), \nabla K*\rho(\cdot,y)(t)\Big) \Big|&&\\
        &\leq C\Big|K*\rho(\cdot,x)(t) - K*\rho(\cdot,y)(t) \Big| + C\Big|\nabla K*\rho(\cdot,x)(t) - \nabla K*\rho(\cdot,y)(t)\Big|&&\\
        &\leq C\int_0^t\int_{\mathbb{R}}\Big|K(x-z)-K(y-z)\Big|\rho(s,z)dzds + C\int_0^t\int_{\mathbb{R}}\Big|\nabla K(x-z)-\nabla K(y-z)\Big|\rho(s,z)dzds&&\\
        &\leq 2C\left(L_K\vee L_K^\prime\right)T|x-y|.
    \end{align*}
\end{proof}

Let us note that if we consider the following equation  $$
\partial_t c = -\lambda   (K*\rho) c,
$$
it admits the following explicit solution in terms of $\rho$
  \begin{equation}\label{eq:c_explicit_regularised}
        c(t,x) \,=\, c_0\exp\left(-\lambda\int_0^t \int_{\mathbb{R}}K(x-y)\rho(s,y)dyds \right),
    \end{equation}
where $c_0$ is its constant initial condition. Hence, de facto, the path dependence of both the advection and reaction term hides the explicit coupling  of the evolution of $\rho$  with $c$ \cite{2024_MTU_arxiv}. However, the coupling is done with a quantity which is regular enough. Precisely, the following result holds.

\begin{proposition}\label{prop:boundedness_continuity_reaction_coeff}
    The function $c$ given by \eqref{eq:c_explicit_regularised} is bounded by  $c_0$ and Lipschitz continuous. Precisely, for any $x,z\in\mathbb{R}\cup\{\Delta\}$ and $t\in[0,T]$,
    \begin{equation*}
        \Big| c(t,x) - c(t,z) \Big| \,\leq\, \lambda c_0L_KT\big|x-z\big|.
    \end{equation*}
\end{proposition}
\begin{proof}
From \eqref{eq:c_explicit_regularised} and \eqref{eq:boundness_convolution}, the boundedness is trivial. Now, let us assume $K(x-y) - K(z-y)\geq 0$. For the case $K(x-y) - K(z-y)< 0$, the same argument holds. Then, since $\Big|e^{-x}-1\Big|\leq|x|$ for every $x\geq 0$,
\begin{equation*}
    \begin{split}
        \left| c(t,x) - c(t,z) \right|& =\,c_0\exp\left(-\lambda\int_0^t \int_{\mathbb{R}}K(z-y)\rho(s,y)dyds \right)\cdot\\
        &\qquad\cdot\Bigg|\exp\left(-\lambda\int_0^t \int_{\mathbb{R}}\Big(K(x-y)-K(z-y)\Big)\rho(s,y)dyds \right) \,-\, 1 \Bigg|\\
        &\leq\, \lambda c_0\int_0^t\int_{\mathbb{R}}\Big| K(x-y) - K(z-y) \Big|\rho(s,y)dyds\,\leq\, \lambda c_0 L_KT|x-z|,
    \end{split}
    \end{equation*}
    where the last inequality is due to \eqref{eq:K_LK}.
\end{proof}

\subsection{Well-posedness of the killed McKean-Vlasov SDE}
This section is devoted to the study of  the existence and uniqueness properties  of the McKean-Vlasov SDE \eqref{eq:killed_SDE_measure}, subject to the killing random time defined in \eqref{eq:def_killing_measure}. To this end, an auxiliary process is introduced, in which the killing mechanism appears only through the law entering the coefficients, and not in the dynamics of the process itself. As shown in Theorem  \ref{theo:killing}, the well-posedness of this auxiliary model implies that of the killed McKean-Vlasov SDE.

The aforementioned auxiliary process is defined as
\begin{equation}\label{eq:SDE_con_killing_solo_dentro}
    \begin{aligned}
        &\widetilde{X}_t = X_0 + \int_0^t b\left(\int_0^s K*\nu_r(\widetilde{X}_s)dr, \int_0^s\nabla K*\nu_r(\widetilde{X}_s)dr\right)ds + \sqrt{2} W_t;\\
        &\widetilde{\Lambda}_t = \int_0^t \lambda c_0\exp\left( -\lambda\int_0^s K*\nu_r(\widetilde{X}_s)dr\right) ds,
    \end{aligned}
\end{equation}
for any $t\in [0,T]$, where $\nu_t=\mathbb P\left( \widetilde{X}_t\in\cdot, \widetilde{\Lambda}_t<Z\right)$.

In light of equations \eqref{eq:subprob_killed_SDE} and \eqref{eq:subprob_killed_SDE_2}, equation \eqref{eq:SDE_con_killing_solo_dentro} can be reformulated by lifting it to an \(\mathbb{R}^2\)-valued process and introducing a functional \(\Phi\) that connects probability measures on \(\mathbb{R}^2\) to sub-probability measures on \(\mathbb{R}\) (see, e.g., \cite[Proposition 3.12]{hambly2023control}).

Specifically, the following map is defined
\begin{equation*}
    \Phi:\mathcal{P}^2\left(\mathbb{R}^2\right)\rightarrow \mathcal{M}_{\leq 1}^2\left(\mathbb{R}\right),\quad \int_{\mathbb{R}}f(x)\Phi({\mu})(dx):=\int_{\mathbb{R}^2} e^{-y}f(x){\mu}(dx,dy),
\end{equation*}
for every bounded continuous test function \(f\), where \(\mathcal{P}^2(\mathbb{R}^2)\) denotes the set of square-integrable probability measures on \(\mathbb{R}^2\), and \(\mathcal{M}_{\leq 1}^2(\mathbb{R})\) the set of square-integrable sub-probability measures on \(\mathbb{R}\).

The process $(\widetilde{X},\widetilde{\Lambda})$ solution to \eqref{eq:SDE_con_killing_solo_dentro} can be rewritten through the map $\Phi$ as
\begin{equation}\label{eq:SDE_con_killing_solo_dentro_lift}
    \begin{aligned}
        &\widetilde{X}_t = X_0 + \int_0^t b\left(\int_0^s K*\Phi(\widetilde{\mu}_r)(\widetilde{X}_s)dr, \int_0^s\nabla K*\Phi(\widetilde{\mu}_r)(\widetilde{X}_s)dr\right)ds + \sqrt{2} W_t;\\
        &\widetilde{\Lambda}_t = \int_0^t \lambda c_0\exp\left( -\lambda \int_0^s K*\Phi(\widetilde{\mu}_r)(\widetilde{X}_s)dr \right)ds,
    \end{aligned}
\end{equation}
for all \(t \in [0,T]\), where \(\widetilde{\mu}\) denotes the law of $(\widetilde{X},\widetilde{\Lambda})$.

The previous reformulation allows working with genuine probability measures on $\mathbb{R}^2$, instead of non-conservative measures on $\mathbb{R}$, thus simplifying the analysis. The well-posedness of equation \eqref{eq:SDE_con_killing_solo_dentro_lift} is addressed in the following.

We will denote with $C^d$ ($C$ for the one-dimensional case) the space of all the real continuous functions from $[0,T]$ to $\mathbb{R}^d$, namely $C\left( [0,T];\mathbb{R}^d\right)$. Let $\mathcal{P}^2\left(C^d\right)$ be the set of probability measures on $C^d$ with finite second moment.  We introduce the following Wasserstein distances.
\begin{definition}\label{def:Wasserstein}
Let $t\in[0,T]$. For any $\mu,\nu\in\mathcal{P}^2\left(C^d\right)$  the \emph{Wasserstein distance} $D^2_t$ between $\mu$ and $\nu$  is defined as
    \begin{equation*}
        D^2_t(\mu,\nu) \,:=\, \inf_{\pi\in\Pi(\mu,\nu)} \int_{C^d\times C^d} \sup_{0\leq s\leq t} \lVert x_s - y_s\rVert^2 \pi\left(dx,dy\right)\,,
    \end{equation*}
    where $\lVert\cdot\rVert$ denotes the Euclidean norm on $\mathbb{R}^d$ and $\Pi(\mu,\nu)$ is the space of coupling of $\mu$ and $\nu$.
\end{definition}

\begin{lemma}\label{lemma:Wasserstein_properties}
    Let $\mu,\nu\in \mathcal{P}^2\left(C^d\right)$. For any $t\leq T$
    \begin{equation*}
        D_t^2(\mu,\nu) \,\leq\, D^2_T(\mu,\nu).
    \end{equation*}
    Also, let $Y^{\mu}$ and $Y^{\nu}$ be two stochastic processes such that $Law\left(Y^{\mu}\right)=\mu$ and $Law\left(Y^{\nu}\right)=\nu$. Then,
    \begin{equation}\label{eq:prop_Wasserstein}
        D_t^2(\mu,\nu) \leq \mathbb{E}\left[ \sup_{0\leq s\leq t} \big\lVert Y_s^{\mu} \,-\, Y_s^{\nu} \big\rVert^2 \right]\,.
    \end{equation}
\end{lemma}
\begin{proof}
    For a complete proof of the results, see e.g. Chapter 6 in \cite{Villani}.
\end{proof}

\begin{proposition}\label{prop:well_posed_SDE}
    The McKean-Vlasov diffusion model \eqref{eq:SDE_con_killing_solo_dentro_lift} admits a strong solution $(\widetilde{X},\widetilde{\Lambda})$ which is pathwise unique. Moreover, it also admits a weak solution which is unique in the sense of probability law.
\end{proposition}
\begin{proof}
    Let $(\widetilde{\mathbf{X}},\widetilde{\mu}):=\left((\widetilde{X},\widetilde{\Lambda}),\widetilde{\mu}\right)$ and $(\widetilde{\mathbf{X}}^{\prime},\widetilde{\mu}^{\prime}):=\left((\widetilde{X}^\prime,\widetilde{\Lambda}^\prime),\widetilde{\mu}^\prime\right)$ be two solutions to equation \eqref{eq:SDE_con_killing_solo_dentro_lift}. By assumption \eqref{eq:K_MK} in the definition of the smooth kernel (Definition~\ref{def:kernel}) and
    \[
    \int_{\mathbb{R}} \Phi(\widetilde{\mu}_u)dx \leq 1, \quad 
    \int_{\mathbb{R}} \Phi(\widetilde{\mu}^{\prime}_u)dx \leq 1 
    \quad \text{for all } u \in [0,T],
    \]
    we obtain the following bounds, valid for every \(s \in [0,T]\):
    \begin{equation}\label{eq:boundedness_3}
        \int_0^s K*\Phi(\widetilde{\mu}_u)(\widetilde{X}_s)\,du,\quad
        \int_0^s K*\Phi(\mu^\prime_u)(\widetilde{X}^\prime_s)\,du \leq M_K T,
    \end{equation}
    and
    \begin{equation}\label{eq:boundedness_4}
        \int_0^s \nabla K*\Phi(\widetilde{\mu}_u)(\widetilde{X}_s)\,du,\quad
        \int_0^s \nabla K*\Phi(\mu^\prime_u)(\widetilde{X}^\prime_s)\,du \leq M_K^\prime T.
    \end{equation}

    Now,
    \begin{equation*}
        \mathbb{E}\left[ \sup_{0\leq a\leq t} \big\lVert \widetilde{\mathbf{X}}_a \,-\, \widetilde{\mathbf{X}}^\prime_a \big\rVert^2 \right] \leq A + B,
    \end{equation*}
    where
    \begin{equation*}
        \begin{aligned}
            A := \mathbb{E}\Bigg[ \sup_{0\leq a\leq t} a \int_0^a \Bigg|&b\left(\int_0^s K*\Phi(\widetilde{\mu}_u)(\widetilde{X}_s)du,\int_0^s \nabla K*\Phi(\widetilde{\mu}_u)(\widetilde{X}_s)du\right)\\
            &-\, b\left(\int_0^sK*\Phi(\mu^\prime_u)(\widetilde{X}^\prime_s)du,\int_0^s\nabla K*\Phi(\mu^\prime_u)(\widetilde{X}^\prime_s)du\right) \Bigg|^2\,ds \Bigg]
        \end{aligned}
    \end{equation*}
    and
    \begin{equation*}
        \begin{aligned}
            B := \lambda c_0\mathbb{E}\Bigg[ \sup_{0\leq a\leq t} a \int_0^a \Bigg|&\exp\left(-\lambda\int_0^s K*\Phi(\widetilde{\mu}_u)(\widetilde{X}_s)du\right)\\
            &- \exp\left(-\lambda\int_0^s K*\Phi(\widetilde{\mu}^\prime_u)(\widetilde{X}^\prime_s)du\right)\Bigg|^2 ds\Bigg].
        \end{aligned}
    \end{equation*}

    As for the first term, from \eqref{eq:boundedness_3} and \eqref{eq:boundedness_4}, by Proposition \ref{prop:continuity_boundedness_b}, and Cauchy-Schwartz inequality,
    \begin{align*}
        A &= \mathbb{E}\Bigg[ t \int_0^t \Bigg| b\left(\int_0^s K*\Phi(\widetilde{\mu}_u)(\widetilde{X}_s)du,\int_0^s \nabla K*\Phi(\widetilde{\mu}_u)(\widetilde{X}_s)du\right) \\
        &\qquad\qquad\quad - b\left(\int_0^sK*\Phi(\mu^\prime_u)(\widetilde{X}^\prime_s)du,\int_0^s\nabla K*\Phi(\mu^\prime_u)(\widetilde{X}^\prime_s)du\right) \Bigg|^2\,ds \Bigg] \\
        &\leq 2CT\mathbb{E}\left[\int_0^t \int_0^s \Bigg| \int_{\mathbb{R}} K(x-\widetilde{X}_s)\Phi(\widetilde{\mu}_u)(dx) - \int_{\mathbb{R}} K(x^\prime-\widetilde{X}^\prime_s)\Phi(\widetilde{\mu}^\prime_u)(dx^\prime) \Bigg|^2\,dud s\right] \\
        &\quad + 2CT\,\mathbb{E}\left[\int_0^t \int_0^s \Bigg| \int_{\mathbb{R}} \nabla K(x-\widetilde{X}_s)\Phi(\widetilde{\mu}_u)(dx) - \int_{\mathbb{R}} \nabla K(x^\prime-\widetilde{X}^\prime_s)\Phi(\widetilde{\mu}^\prime_u)(dx^\prime) \Bigg|^2\,dud s\right] \\
        &=: A_1 + A_2\,.
    \end{align*}

    Let $\pi\in\Pi(\widetilde{\mu},\widetilde{\mu}^\prime)$ be a coupling of $\widetilde{\mu}$ and $\widetilde{\mu}^\prime$. By Jensen inequality, the Lipschitz continuity of $K$ and $\nabla K$, and Cauchy-Schwartz inequality, the first term $A_1$ can be bounded as follows.
    \begin{align*}
        A_1 &= 2CT\mathbb{E}\left[ \int_0^t\int_0^s \Bigg| \int_{C^2} K(x_u-\widetilde{X}_s)e^{-y_u}\widetilde{\mu}(dx,dy) -\int_{C^2} K(x_u^\prime-\widetilde{X}^\prime_s)e^{-y_u^\prime}\widetilde{\mu}^\prime(dx^\prime,dy^\prime) \Bigg|^2duds \right]\\
        &\leq 2CT\mathbb{E}\left[ \int_0^t\int_0^s\int_{C^2\times C^2} \bigg| K(x_u-\widetilde{X}_s)e^{-y_u} -K(x_u^\prime-\widetilde{X}^\prime_s)e^{-y_u^\prime} \bigg|^2\pi(d\mathbf{x},d\mathbf{x}^\prime)duds \right]\\
        &\leq 2CT\mathbb{E}\left[ \int_0^t\int_0^s\int_{C^2\times C^2} e^{-2y_u^\prime}\bigg| K(x_u-\widetilde{X}_s)\exp(|y_u-y_u^\prime|)-K(x_u^\prime-\widetilde{X}^\prime_s) \bigg|^2\pi(d\mathbf{x},d\mathbf{x}^\prime)duds \right]\\
        &\leq 2CT\mathbb{E}\left[ \int_0^t\int_0^s\int_{C^2\times C^2} \bigg| K(x_u-\widetilde{X}_s)(1+|y_u-y_u^\prime|) -K(x_u^\prime-\widetilde{X}^\prime_s) \bigg|^2\pi(d\mathbf{x},d\mathbf{x}^\prime)duds \right]\\
        &\leq 8CT\int_0^t\int_0^s\int_{C^2\times C^2} |\mathbf{x}_u-\mathbf{x}_u^\prime|^2\pi(d\mathbf{x},d\mathbf{x}^\prime)duds + 8CT^2\mathbb{E}\left[ \int_0^t |\widetilde{X}_s-\widetilde{X}^\prime_s|^2ds \right].
    \end{align*}

    For the terms  $A_2$ and $B$, similar arguments may be applied. Therefore, by Fubini's theorem,
    \begin{equation*}
        \begin{aligned}
            &\mathbb{E}\left[ \sup_{0\leq a\leq t} \big\lVert \widetilde{\mathbf{X}}_a \,-\, \widetilde{\mathbf{X}}^\prime_a \big\rVert^2 \right]\\
            &\leq 8CT^2\int_0^t\int_{C^2\times C^2} \sup_{0\leq u\leq s}\lVert\mathbf{x}_u-\mathbf{x}_u^\prime\rVert^2\pi(d\mathbf{x},d\mathbf{x}^\prime)duds + 8CT^2\int_0^t\mathbb{E}\left[  \lVert\widetilde{\mathbf{X}}_s-\widetilde{\mathbf{X}}^\prime_s\rVert^2 \right]ds.
        \end{aligned}
    \end{equation*}
    
    Taking the infimum over $\Pi(\widetilde{\mu},\widetilde{\mu}^{\prime})$ on both sides of the above inequality, we obtain
    \begin{equation}\label{eq4:proof_well_posed_SDE}
        \mathbb{E}\left[ \sup_{0\leq a\leq t} \big\lVert \widetilde{\mathbf{X}}_a \,-\, \widetilde{\mathbf{X}}^\prime_a \big\rVert^2 \right] \leq 8CT^2\int_0^tD_s^2(\widetilde{\mu},\widetilde{\mu}^{\prime})ds + 8CT^2\int_0^t\mathbb{E}\left[  \lVert\widetilde{\mathbf{X}}_s-\widetilde{\mathbf{X}}^\prime_s\rVert^2 \right]ds.
    \end{equation}
    
    Applying Gronwall's lemma to \eqref{eq4:proof_well_posed_SDE} yields
    \begin{equation}\label{eq5:proof_well_posed_SDE}
        \mathbb{E}\left[ \sup_{0\leq a\leq t} \big\lVert \widetilde{\mathbf{X}}_a \,-\, \widetilde{\mathbf{X}}^\prime_a \big\rVert^2 \right] \leq  8CT^2e^{8CT^2t} \int_0^t D_s^2(\widetilde{\mu},\widetilde{\mu}^{\prime})ds.
    \end{equation}
    
    By \eqref{eq:prop_Wasserstein} of Lemma \ref{lemma:Wasserstein_properties},
    \begin{equation}\label{eq6:proof_well_posed_SDE}
        D^2_t(\widetilde{\mu},\widetilde{\mu}^{\prime}) \leq C_T\int_0^t D_r^2(\widetilde{\mu},\widetilde{\mu}^{\prime})dr,
    \end{equation}
    where $C_T = C_T\left(\varphi_0,\varphi_1,\lambda,c_0,M_K',T\right)$. Therefore, by Gronwall's lemma, we obtain
    \begin{equation}\label{eq7:proof_well_posed_SDE}
        D^2_T(\widetilde{\mu},\widetilde{\mu}^{\prime}) = 0.
    \end{equation}
    
    Finally, by substituting \eqref{eq7:proof_well_posed_SDE} in \eqref{eq5:proof_well_posed_SDE}, we get
    \begin{equation*}
        \mathbb{E}\left[ \sup_{0\leq a\leq t} \big\lVert \widetilde{\mathbf{X}}_a \,-\, \widetilde{\mathbf{X}}^\prime_a \big\rVert^2 \right] = 0,
    \end{equation*}
    which proves pathwise uniqueness.

    To prove the existence of a weak solution to SDE \eqref{eq:SDE_con_killing_solo_dentro_lift}, let us introduce the map
    \begin{equation*}
    \Theta:\mathcal{P}^2\left(C^2\right)\rightarrow\mathcal{P}^2\left(C^2\right)\quad \textit{s.t.}\quad \Theta(\widetilde{\mu})\,=\, \mathcal{L}(\widetilde{\mathbf{X}})\,,
        \end{equation*}
    where $\widetilde{\mathbf{X}}$ is a solution to equation \eqref{eq:SDE_con_killing_solo_dentro_lift}. As $\left(\mathcal{P}^2\left(C^2\right),D^2_T\right)$ is a complete separable metric space, it just remains to show the existence of fixed point for the map $\Theta$.
    
    By the definition of $\Theta$, we can rewrite \eqref{eq6:proof_well_posed_SDE} as
    \begin{equation}\label{eq8:proof_well_posed_SDE}
        D^2_T\left(\Theta(\widetilde{\mu}),\Theta(\widetilde{\mu}^\prime)\right) \,=\, D^2_T(\widetilde{\mu},\widetilde{\mu}^\prime) \leq C_T\int_0^T D_r^2(\widetilde{\mu},\widetilde{\mu}^\prime)dr\,.
    \end{equation}
    
    Now, let us define the sequence
    \begin{equation*}
        \{m_k\}_{k\in \mathbb{N}} \,:=\, \{\Theta^k(m)\}_{k\in\mathbb{N}},
    \end{equation*}
    where $m\in \mathcal{P}^2\left(C^2\right)$.
    
    By definition of $m_k$ and \eqref{eq8:proof_well_posed_SDE},
    \begin{equation*}
        D^2_T\left( \Theta^{k+1}(m),\Theta^k(m) \right) \,\leq\, C_T\int_0^T D_u^2\left(\Theta^k(m),\Theta^{k-1}(m)\right)du.
    \end{equation*}
    Iterating the above inequality,
    \begin{equation*}
        D_T^2\left( \Theta^{k+1}(m),\Theta^k(m) \right) \,\leq\, \frac{C_T^k T^k}{k!}\,D_T^2\left( \Theta(m),m \right).
    \end{equation*}
    Since
    \begin{equation*}
        \frac{C_T^k T^k}{k!} \,\rightarrow\, 0\quad \textit{as}\quad k\,\rightarrow\,+\infty,
    \end{equation*}
    $\{m_k\}_{k\in\mathbb{N}}$ is a Cauchy sequence. Therefore, there exists a probability measure $\mathbb{Q}$ such that $m_k$ weakly converges to  $\mathbb{Q}$ as $k\rightarrow +\infty$ and $\Theta(\mathbb{Q}) = \mathbb{Q}$, by definition of $m_k$. By construction of $\Theta$, this implies that there exists a weak solution to SDE \eqref{eq:SDE_con_killing_solo_dentro_lift}.

    Finally, by the results of Yamada and Watanabe (see   \cite[308]{Karatzas}), we can conclude that there exists a pathwise unique, strong solution to SDE \eqref{eq:SDE_con_killing_solo_dentro_lift}.

    As for the uniqueness in law, we follow the proofs in \cite[Theorem 3.1]{2016_Russo} and  \cite[Proposition 3.3]{2024_MTU_arxiv}, applied to the present case.
    Again, let $(\widetilde{\mathbf{X}},\widetilde{\mu})$ and $(\widetilde{\mathbf{X}}^\prime,\widetilde{\mu}^\prime)$ be two solutions to equation \eqref{eq:SDE_con_killing_solo_dentro_lift}  on possibly different probability spaces, Brownian motions, and initial conditions distributed according to $\zeta_0$, such that, $\mathcal{L}(\zeta_0)=\mu_0$ and $\mu_0(dx)=\rho_0(x)dx$. As usual, $\rho_0(\cdot)$ is the initial condition to the PDE model \eqref{eq:PDE_rho_regularised},\eqref{eq:def_b_PDE}.

    Given $\nu\in\mathcal{P}^2\left(C^2\right)$, we indicate by $\Theta(\nu)$ the law of $\overline{\mathbf{X}}$, where $\overline{\mathbf{X}}=(\overline{X},\overline{\Lambda})$ is the strong solution of
    \begin{equation}\label{eq:Y_bar}
        \begin{aligned}
            &\overline{X}_t = X_0 + \int_0^t b\left(\int_0^s K*\Phi(\widetilde{\mu}^\prime_r)(\overline{X}_s)\,dr,\ \int_0^s\nabla K*\Phi(\widetilde{\mu}^\prime_r)(\overline{X}_s)dr\right)ds + \sqrt{2} W_t;\\
            &\overline{\Lambda}_t = \int_0^t \lambda c_0\exp\left( -\lambda \int_0^s K*\Phi(\widetilde{\mu}^\prime_r)(\overline{X}_s)dr \right)ds,
        \end{aligned}
    \end{equation}
    on the same probability space and same Brownian motion on which $\mathbf{X}$ lives.

    Since $\widetilde{\mu}^\prime$ is fixed, $\overline{\mathbf{X}}$ is solution of a classical SDE for which pathwise uniqueness holds. In fact, the boundedness (see \eqref{eq:bound_drift} in Proposition \ref{prop:continuity_boundedness_b}) and  Lipschitz continuity (see Proposition \ref{prop:lipschitz_prop_b}) of the drift coefficient $b$  assure that there exists a pathwise unique strong solution $\overline{\mathbf{X}}$ to \eqref{eq:Y_bar}. By Yamada-Watanabe theorem, $\widetilde{\mathbf{X}}^\prime$ and $\overline{\mathbf{X}}$ have the same distribution. Consequently, $\Theta(\widetilde{\mu}^\prime)=\mathcal{L}\left( \overline{\mathbf{X}} \right)=\mathcal{L}\left( \widetilde{\mathbf{X}}^\prime \right)=\widetilde{\mu}^\prime$.

    Finally, we prove that $\mathcal{L}(\widetilde{\mathbf{X}})=\mathcal{L}\left(\overline{\mathbf{X}}\right)$, i.e. $\widetilde{\mu}=\widetilde{\mu}^\prime$. From \eqref{eq8:proof_well_posed_SDE}, we get
    \begin{equation*}
        \forall t\in[0,T]\quad\quad D^2_t\left(\Theta(\widetilde{\mu}),\Theta(\widetilde{\mu}^\prime)\right) \,\leq\, C_T\,\int_0^t D_r^2(\widetilde{\mu},\widetilde{\mu}^\prime)\,dr
    \end{equation*}
    Since $\Theta(\widetilde{\mu})=\widetilde{\mu}$ and $\Theta(\widetilde{\mu}^\prime)=\widetilde{\mu}^\prime$, by Gronwall's lemma we conclude.
\end{proof}

\begin{theorem}\label{theo:killing}
    The McKean-Vlasov diffusion model \eqref{eq:killed_SDE_measure},\eqref{eq:def_killing_measure} admits a strong solution $X$ which is pathwise unique. Moreover, it also admits a weak solution which is unique in the sense of probability law.
\end{theorem}
\begin{proof}
    The original killed McKean–Vlasov process \eqref{eq:killed_SDE_measure} evolves in \(\mathbb{R}\) only up to the stopping time $\tau$ \eqref{eq:def_killing_measure}. However, we have established the well-posedness of a related lifted system, in which the killing mechanism is incorporated exclusively in the law of the process appearing in the coefficients, not in the dynamics of the process itself. Since this lifted system is well-posed on the entire interval \([0,T]\), the well-posedness of the original killed process, i.e. the process which incorporates killing also in its dynamics, is then a direct consequence.
\end{proof}

\section{A probabilistic interpretation of the  PDE model}\label{sec:probabilistic_interpretation}
This section represents the core of the probabilistic interpretation of the PDE \eqref{eq:PDE_rho_regularised},\eqref{eq:def_b_PDE} . The main objective is to prove that the time marginal $\nu_t$ admits a density $v(t,\cdot)$, which in turn satisfies the PDE \eqref{eq:PDE_rho_regularised},\eqref{eq:def_b_PDE}. To this end, the existence of the density is first established, and Itô’s formula is then applied to the solution of the killed McKean–Vlasov equation \eqref{eq:killed_SDE_measure},\eqref{eq:def_killing_measure} to derive the PDE satisfied by \(v\).

The family of sub-probability measures $\{\nu_t\}_{t\in[0,T]}$ as in \eqref{eq:subprob_killed_SDE} is absolutely continuous with respect to the Lebesgue measure. Indeed, following  \cite{Meleard_Coppoletta,Tomasevic}, we may prove the following result.
\begin{theorem}\label{theo:density}
    Let $ C_c^{\infty}:=C_c^{\infty}(\mathbb{R})$ be the space of $C^{\infty}$ real-valued functions with compact support. For any $t\in (0,T]$, the sub-probability measure $\nu_t$ defined in \eqref{eq:subprob_killed_SDE} admits a probability density function $v(t,\cdot)$. Moreover, for any $p\in (1,+\infty)$, $v(t,\cdot)\in L^p(\mathbb{R})$ uniformly in $[\varepsilon,T]$, with $\varepsilon>0$.
\end{theorem}
\begin{proof}
    The existence of a density probability function is proven by applying in a standard way Riesz theorem in $L^p$ spaces. Let  $p>1$ be fixed, and $q$ be such that $\frac{1}{p} \,+\, \frac{1}{q} =1$.
    
    Let $\left(C_c^{\infty}\right)^*$ be the dual space of $C_c^{\infty}$. For any  $t\in[0,T]$, we can  define a linear functional $H_t(\cdot)\in \left(C_c^{\infty}\right)^*$ as  
    \begin{equation*}
         H_t(f) :=\int_{\mathbb{R}} f(x)\nu_t(dx) = \mathbb{E}\left[ f(\widetilde{X}_t)\exp\left(-\int_0^t \lambda c_0\exp\left( -\lambda \int_0^s K*\nu_r(\widetilde{X}_s)dr \right)ds \right) \right],
    \end{equation*}
   for any $f\in C_c^{\infty}$, where $\widetilde{X}$ is a solution to \eqref{eq:SDE_con_killing_solo_dentro}.

    Fix \( t \in (0,T] \). If \( H_t \) satisfies the continuity bound  
    \begin{equation}\label{eq:continuity_H_t_Lq}
        |H_t(f)| \leq \overline{C} \|f\|_{L^q},
    \end{equation}
    for all \( f \in C_c^\infty(\mathbb{R}) \), then it can be extended continuously to the whole space \( L^q(\mathbb{R}) \), since \( C_c^\infty(\mathbb{R}) \) is dense in \( L^q(\mathbb{R}) \). Indeed, by the Hahn--Banach theorem, there exists a continuous linear extension \( \widetilde{H}_t \in (L^q)^* \) such that \( \widetilde{H}_t(f) = H_t(f) \) for all \( f \in C_c^\infty(\mathbb{R}) \).
    
    By the Riesz representation theorem (see, e.g., \cite{Brezis}), the above functional \( \widetilde{H}_t \) can be represented by a unique function \( v(t, \cdot) \in L^p(\mathbb{R}) \), with \( 1/p + 1/q = 1 \), such that
    \begin{equation*}
        H_t(f) = \int_{\mathbb{R}} f(x)\nu_t(dx) = \int_{\mathbb{R}} f(x)v(t,x)\,dx, \quad \text{for all } f \in C_c^\infty(\mathbb{R}).
    \end{equation*}
    This means that the measure \( \nu_t \) admits a density \( v(t,\cdot) \in L^p(\mathbb{R}) \), and its norm satisfies
    \begin{equation*}
        \|v(t,\cdot)\|_{L^p(\mathbb{R})} = \|\widetilde{H}_t\|_{(L^q)^*} = \sup_{f \in L^q(\mathbb{R})} \frac{|H_t(f)|}{\|f\|_{L^q}}.
    \end{equation*}

    We now prove inequality \eqref{eq:continuity_H_t_Lq}.
    To start, let us notice that
    \begin{equation*}
        H_t(f) \leq \mathbb{E}\left[f(\widetilde{X}_t)\right].
    \end{equation*}
    Now, let $\mathbb{Q} $ be the absolutely continuous measure  with respect to $\mathbb P$ with density $Z_T$ given by
    \begin{equation*}
        \begin{aligned}
            Z_T  := \exp\Bigg(&-\int_0^T \frac{1}{\sqrt{2}}b\left(\int_0^sK*\nu_r(\widetilde{X}_s)dr,\int_0^s\nabla K*\nu_r(\widetilde{X}_s)dr\right)dW_s\\
            &-\,\frac{1}{4}\int_0^T\Bigg|b\left(\int_0^sK*\nu_r(\widetilde{X}_s)dr,\int_0^s\nabla K*\nu_r(\widetilde{X}_s)dr\right)\Bigg|^2ds\Bigg).
        \end{aligned}
    \end{equation*}
    Since by Proposition \ref{prop:continuity_boundedness_b} the  drift is bounded in the finite time $T$,  by the  Novikov  condition, $Z_T$ is a martingale. Then, by Girsanov theorem, $\frac{1}{\sqrt{2}} \widetilde{X}$ is a Wiener process under $\mathbb{Q}$ and the linear operator may be expressed in term of $Z_T$ as
    \begin{equation}\label{eq1:proof_density}
        H_t(f) \leq \mathbb{E}^{\mathbb{P}}\left[ f\left(\sqrt{2}W_t \,+\, X_0\right)(Z_T)^{-1} \right].
    \end{equation}
Now, let us choose $p_2\in(1,q)$ and $q_2$ such that $\frac{1}{p_2} \,+\, \frac{1}{q_2} =1$. By applying Hölder's inequality to \eqref{eq1:proof_density} with $p_3 = \frac{q}{p_2}$ and $q_3$ such that $\frac{1}{p_3} \,+\, \frac{1}{q_3} =1$, we have
    \begin{equation*}
        |H_t(f)| \,\leq\, \underbrace{\mathbb{E}\left[|f(X_0 \,+\, \sqrt{2}W_t)|^{p_3}\right]^{\frac{1}{p_3}}}_{:=\,\Theta}\underbrace{\mathbb{E}\left[(Z_T)^{-q_3}\right]^{\frac{1}{q_3}}}_{:=\,\Xi}.\\        
    \end{equation*}
    Let us estimate $\Theta$ and $\Xi$. By Hölder's inequality the following estimates in terms of the $L^q$ norm of $f$ can be obtained
    \begin{align*}
        \Theta \,&=\, \left( \int_{\mathbb{R}}\int_{\mathbb{R}} |f(x+\sqrt{2}y)|^{p_3}\,p(t,y)dy\rho_0(x)dx \right)^{\frac{1}{p_3}}&&\\
        &\leq\,\left( \int_{\mathbb{R}} \left(\int_{\mathbb{R}} |f(x+\sqrt{2}y)|^{q}dy\right)^{\frac{1}{p_2}}\, \lVert p(t,\cdot)\rVert_{L^{q_2}(\mathbb{R})}\rho_0(x)dx \right)^{\frac{1}{p_3}}&&\\
        &=\,\left( \lVert f \rVert_{L^{q}(\mathbb{R})}^{\frac{q}{p_2}} \lVert p(t,\cdot) \rVert_{L^{q_2}(\mathbb{R})} \right)^{\frac{p_2}{q}} \,=\, C\,\,t^{-\frac{1}{2 q }  }
        \lVert f \rVert_{L^{q}(\mathbb{R})},&&
    \end{align*}
    where $p(t,\cdot)$ is the Gaussian probability density   and  
    \begin{equation*}
     \lVert p(t,\cdot) \rVert_{L^{q_2}(\mathbb{R})} = \left(\int_{\mathbb R}\left|\frac{1}{\sqrt{2\pi t}}\exp\left( -\frac{x^2}{2t} \right) \right|^{q_2}dx\right)^{\frac{1}{q_2}}={\frac{(2\pi)^{\frac{1}{2}\left(\frac{1}{q_2}-1\right)}}{q_2^{1/2} }}  t^{-\frac{1}{2 }\left(1-\frac{1}{q_2}\right)}
       = C\,\,t^{-\frac{1}{2p_2 }  }.
    \end{equation*}
     
    On the other hand,  for the $\Xi$ term, by Cauchy-Schwartz inequality, the properties of the exponential martingale properties, and \eqref{eq:bound_drift} we get 
    \begin{eqnarray*}
        \Xi^{q_3} &=& \mathbb{E}\Bigg[\exp\Bigg\{ \frac{q_3}{\sqrt{2}}\int_0^T b\left(\int_0^sK*\nu_r(\widetilde{X}_s)dr,\int_0^s\nabla K*\nu_r(\widetilde{X}_s)dr \right)dW_s\\
        && \hspace{2cm}-\,\frac{q_3^2}{4}\int_0^T\Bigg| b\left(\int_0^sK*\nu_r(\widetilde{X}_s)dr,\int_0^s\nabla K*\nu_r(\widetilde{X}_s)dr \right)\Bigg|^2ds\\
        &&\hspace{2cm}+\,\left(\frac{q_3^2}{4} \,+\, \frac{q_3}{4}\right)\int_0^T\Bigg|b\left(\int_0^sK*\nu_r(\widetilde{X}_s)dr,\int_0^s\nabla K*\nu_r(\widetilde{X}_s)dr \right)\Bigg|^2ds\Bigg\} \Bigg]\\
        &\leq&\mathbb{E}\Bigg[\exp\Bigg\{ \frac{2q_3}{\sqrt{2}}\int_0^T b\left(\int_0^sK*\nu_r(\widetilde{X}_s)dr,\int_0^s\nabla K*\nu_r(\widetilde{X}_s)dr \right)dW_s\\
        &&\hspace{2cm}-\frac{q_3^2}{2}\int_0^T\Bigg| b\left(\int_0^sK*\nu_r(\widetilde{X}_s)dr,\int_0^s\nabla K*\nu_r(\widetilde{X}_s)dr \right)\Bigg|^2ds\Bigg\}\Bigg]^{1/2}\\
        &&\cdot\mathbb{E}\Bigg[\exp\Bigg\{\left(\frac{q_3^2}{2} \,+\, \frac{q_3}{2}\right)\int_0^T\Bigg|b\left(\int_0^sK*\nu_r(\widetilde{X}_s)dr,\int_0^s\nabla K*\nu_r(\widetilde{X}_s)dr \right)\Bigg|^2ds\Bigg\} \Bigg]^{1/2}\\
        &\leq&\exp\Bigg\{\left(\frac{q_3^2}{2} \,+\, \frac{q_3}{2}\right)TM_b^2\Bigg\}=:C^\prime(T,M_b).
       \end{eqnarray*} 
    
   Thus, the continuity properties \eqref{eq:continuity_H_t_Lq}  with $ \overline{C} = C C^\prime(T,M_b) t^{-\frac{1}{2q}}=\widetilde{C}t^{-\frac{1}{2q}}$ and the existence of a density $u_t$ for $\nu_t$ are proven. Furthermore, for $t>\epsilon$
    \begin{equation*}
        \lVert v(t,\cdot) \rVert_{L^{p}(\mathbb{R})} = \lVert H_t \rVert_{\left(L^{q} \left(\mathbb{R}\right)\right)^*} = \sup_{f\in L^{q}\left(\mathbb{R}\right)}\frac{\left|H_t(f)\right|}{\lVert f\rVert_{L^{q}\left(\mathbb{R}\right)}} \leq \widetilde{C}t^{-\frac{1}{2q}}, 
    \end{equation*}
    i.e. $u(t,\cdot)\in L^p$ and $ \sup_{t\in[\varepsilon,T]}\lVert u(t,\cdot) \rVert_{L^p(\mathbb{R})} < \infty$.
\end{proof}

By applying Ito’s formula to the solution of the McKean–Vlasov SDE with killing \eqref{eq:killed_SDE_measure},\eqref{eq:def_killing_measure}, we can characterize the density $v$ as a weak solution to the PDE \eqref{eq:PDE_rho_regularised},\eqref{eq:def_b_PDE}.
\begin{theorem}\label{theo:prob_interpretation}
    The density $v(t,\cdot)$ of the sub-probability measure $\nu_t$ defined in \eqref{eq:subprob_killed_SDE} satisfies
    \begin{align*}
        \int_{\mathbb{R}}f(x)v(t,x)dx = &\int_{\mathbb{R}}f(x)\rho_0(x)dx \,+\, \int_0^t\int_{\mathbb{R}}\Delta f(x)v(s,x)dxds&&\\
        &+\int_0^t\int_{\mathbb{R}}\nabla f(x)b\left(K*v(\cdot,x)(s),\nabla K*v(\cdot,x)(s)\right)v(s,x)dxds&&\\
        &-\int_0^t \int_{\mathbb{R}}  f(x)\lambda c_0\exp\big( -\lambda K*v(\cdot,x)(s)\big)v(s,x)dxds,
    \end{align*}
    which reads as a weak formulation of  the  PDE \eqref{eq:PDE_rho_regularised},\eqref{eq:def_b_PDE}.
\end{theorem}
\begin{proof}
    Let $f\in C_b^2\left(\mathbb{R}\right)$ and \(t\in[0, T]\). By applying Ito's formula to the process \(f(X_t)\),
    \begin{equation*}
        \begin{aligned}
            f\left(X_t\right) \,=\, &f\left(X_0\right) - f\left(X_{\tau}\right)\ind{t\geq \tau} + \int_0^t \Delta f(X_s)ds\\
            &+ \int_0^t \nabla f(X_s)b\left(\int_0^s K*\nu_r(X_s)dr,\int_0^s \nabla K*\nu_r(X_s)dr\right)ds&&\\
            &+ \sqrt{2}\int_0^t\nabla f(X_s)dW_s.
        \end{aligned}
    \end{equation*}
    By \eqref{eq:def_poisson} and \eqref{eq:def_compensator_poisson}, given the zero mean martingale $\{N_t - \Lambda_t\}_t$, the following decomposition holds
    \begin{equation*}
        \begin{aligned}
            &f\left(X_{\tau}\right)\ind{t\geq \tau} = \int_0^t f\left(X_{s-}\right)dN_s\\
            &=\int_0^t f(X_{s-}) d\left(N_s - \Lambda_s\right) +  
            \int_0^t f\left(X_{s}\right)\lambda c_0\exp\left( -\lambda \int_0^s K*\nu_r(X_s)dr \right)ds.
            \end{aligned}
    \end{equation*}
    Therefore,
    \begin{equation*}
        \begin{aligned}
            f\left(X_t\right) \,=\, &f\left(X_0\right) + \int_0^t \Delta f(X_s)ds\\
            &+\int_0^t \nabla f(X_s)b\left(\int_0^s K*\nu_r(X_s)dr,\int_0^s \nabla K*\nu_r(X_s)dr \right)ds&&\\
            &-\int_0^t f\left(X_s\right)\lambda c_0\exp\left( -\lambda \int_0^s K*\nu_r(X_s)dr\right)ds&&\\
            &+ \sqrt{2}\int_0^t\nabla f(X_s)dW_s-\int_0^t f(X_{s-}) d\left(N_s - \Lambda_s\right),
        \end{aligned}
    \end{equation*}
     
    Taking the expectation, we get
    \begin{align*}
        \mathbb{E}\left[f\left(X_t\right)\right] \,=\, &\mathbb{E}\left[f\left(X_0\right)\right] + \int_0^t \mathbb{E}\left[\Delta f(X_s)\right]ds\\
        &+\int_0^t \mathbb{E}\left[\nabla f(X_s)b\left(\int_0^s K*\nu_r(X_s)dr,\int_0^s \nabla K*\nu_r(X_s)dr\right)\right]ds&&\\
        &-\int_0^t \mathbb{E}\left[f\left(X_s\right)\lambda c_0\exp\left( -\lambda \int_0^s K*\nu_r(X_s)dr\right)\right]ds.&&
    \end{align*}
    By Theorem \ref{theo:density}, for any $t\in (0,T]$ $\nu_t$ has density $v(t,\cdot)$. Hence, we may rewrite the above equation as
    \begin{align*}
        \int_{\mathbb{R}}f(x)v(t,x)dx = &\int_{\mathbb{R}}f(x)\rho_0(x)dx \,+\, \int_0^t\int_{\mathbb{R}}\Delta f(x)v(s,x)dxds&&\\
        &+\int_0^t\int_{\mathbb{R}}\nabla f(x)b\left(K*v(\cdot,x)(s),\nabla K*v(\cdot,x)(s)\right)v(s,x)dxds&&\\
        &-\int_0^t \int_{\mathbb{R}}  f(x)\lambda c_0\exp\big( -\lambda K*v(\cdot,x)(s)\big)v(s,x)dxds.
    \end{align*}
    Therefore, $\{v(t,\cdot)\}_{t\in[0,T]}$ is a weak solution to the equation \eqref{eq:PDE_rho_regularised},\eqref{eq:def_b_PDE}.
\end{proof}

\section{The associated finite particle system}\label{sec:PS_and_chaos}
A system of Brownian particles, which live and interact up to their survival time, can be naturally associated to  equations \eqref{eq:killed_SDE_measure},\eqref{eq:def_killing_measure}.  Note that the empirical measure of the alive particles, that are the particle in the system which have not yet reacted, in such a particle system approximates the sub-probability law \eqref{eq:subprob_killed_SDE_2} of the killed McKean–Vlasov process.

Let $N\in\mathbb{N}$. Consider the product probability space $\left( \Omega^N,\mathcal{F}^{\otimes N}, \mathbb{P}^{\otimes N} \right)$ filtered by a natural extension of the original filtration to the product space, and an $N$-dimensional Brownian motion $\{\mathbf{W}_t\}_{t\in[0,T]}$ adapted to it. Let $\{X_0^i\}_{i=1}^N$ be a family of $\mathcal{F}_0$-measurable and $\mathbb{R}$-valued random variables which are independent and identically distributed according to $\zeta_0$, where $\zeta_0$ is as in \eqref{eq:inital_datum_SDE}. Let us also assume that $X_0^i$ is independent of $W^j$ for every $i,j\in\{1,\dots,N\}$.

Let $t\in I(\tau^i,T):=[0,\tau^i)\cap [0, T]$, where $\tau^i$ is an $\mathbb{F}$-stopping time, and $i=1,\dots,N$. The dynamics of the $i$-th particle is an $\mathbb{R}\cup\{\Delta\}$-valued process satisfying the equation
\begin{equation}\label{eq:particle_system}
    \begin{aligned}
        &\xi_t^i \,=\, \xi_0^i \,+\, \int_0^t b\left( u_N\left(\cdot,\xi_s^i\right)(s),\nabla u_N\left(\cdot,\xi_s^i\right)(s) \right)ds \,+\, \sqrt{2}W_t^i,\quad &t\in I(\tau^i,T);\\
        &\xi_t^i \in \{\Delta\},\quad &t\geq \tau^i,
    \end{aligned}
\end{equation}
where $\xi_0^i \,=\, X_0^i$, with the convention $\xi_{\tau^i}^i:=\lim_{t\rightarrow\tau^i-}\xi_t^i$.

As already mentioned, we call \emph{alive} the particles which have not yet reacted and denote by   $\Gamma^N_t\subseteq\{1,\dots,N\}$ the subset of indexes corresponding to alive particles  out of $N\in \mathbb N$ at time $t$. The empirical measure and empirical density of the particle system \eqref{eq:particle_system} are
\begin{eqnarray}\label{eq:empirical_measure_mu}
\mu_t^N (\cdot):=\frac{1}{N}\sum_{i\in\Gamma_t^N} \varepsilon_{\xi_t^i}(\cdot)=\frac{1}{N}\sum_{i=1}^N \varepsilon_{\xi_t^i}(\cdot)\mathbbm{1}_{(t,\infty)}(\tau^i)\end{eqnarray}
and
$$u_N\left(t,x\right):= \left(K*\mu_t^N\right)(x) := \int_{\mathbb{R}}K(x-y)\mu_t^N(dy),$$
respectively, while the operators $u_N\left(\cdot,\xi_s^i\right)(s)$ and $\nabla u_N\left(\cdot,\xi_s^i\right)(s)$ are defined as in \eqref{eq:def_rho_integral_and_convolution}. The stopping time $\tau^i$ is the lifetime of the $i$-th particle defined as
\begin{equation}\label{eq:stopping_time_PS}
    \tau^i = \tau^i\left(\xi^i\right) \,:=\, \inf_{t\geq 0}\Bigg\{ \int_0^t \lambda c_0\exp\left(-\lambda u_N(\cdot,\xi_s^i)(s) \right)ds \,\geq\, Z_i \Bigg\},
\end{equation}
where $Z_i\sim Exp(1)$ is independent of $\xi^{1,N},\dots,\xi^{N,N}$.

\begin{remark}\label{rem:Poisson_and_PS}
    The particles are exchangeable by definition. Also, let $\{N^{0,i}\}_{i=1}^N$ be a family of standard Poisson processes, such that for every $i,j = 1, \dots, N$, the process $N^{0,i}$ is independent from $W^j$ and $X_0^j$. Then, the killing time $\tau^i$ of the $i$-th particle is defined as the first and only jump time of a process $N^i_t$, obtained by applying a suitable time change to $N^{0,i}$. That is,
    \begin{equation*}
        N_t^i := N_{\Lambda_t^i}^{0,i},
    \end{equation*}
    where
    \begin{equation*}
        \Lambda_t^i := \int_0^t \lambda c_0\exp\left(-\lambda u_N(\cdot,\xi_s^i)(s) \right)ds
    \end{equation*}
    is the compensator of the process $N_t^i$ itself. Accordingly, the sub-probability empirical law may be expressed as
    \begin{equation}\label{eq:empirical_measure_mu_Lambda}
        \mu_t^N (\cdot):=\frac{1}{N} \sum_{i=1}^N \varepsilon_{(\xi_t^i,\Lambda_t^i)}\left(\cdot \times (0,Z^i)\right),
    \end{equation}
    where $\{Z^i\}_{i=1}^N$ is a family of i.i.d. random variables, with $Z^i\sim \exp(1)$ for every $i=1,\dots,N$, and independent of $\{W^1,...,W^N\}$. Note that the empirical measure \eqref{eq:empirical_measure_mu_Lambda} may be seen as the natural estimator of the bivariate process \eqref{eq:SDE_con_killing_solo_dentro}.
\end{remark}

We end the study with  the analysis of the particle system \eqref{eq:particle_system} and the proof that it is well-posed and has a unique strong solution. Precisely, we have the following result.
\begin{proposition}\label{prop:well_posed_PS}
   Let $T>0$ and $N\in\mathbb N$ fixed.  The particle system \eqref{eq:particle_system} with drift given by \eqref{eq:def_b_PDE} admits a pathwise unique strong solution $\xi_t=\left(\xi_1^t,\ldots,\xi_N^t \right), t\in [0,T]$.
\end{proposition}
\begin{proof}
    For any  fixed $T>0$, let us consider any larger time $S>T$. We consider a constructive proof of the existence of a solution of the particle system \eqref{eq:particle_system} in $[0,S)$, so that the thesis is obtained in $[0,T]$. Well-posedness in $[0,S)$ is achieved  by means of a constructive argument as in \cite[Section 3.1]{hambly2023control}. 

     For \( n = 0 \), define \( \sigma_0 = 0 \), and let \( \Gamma_0^N = \{1, \dots, N\} \) be the set of indices corresponding to the particles alive at time \( \sigma_0 \), that is, the entire collection of particles. Moreover, set \( \xi^{0,i} = \xi_0^i \) for each \( i \in \Gamma_0^N \), where \( \xi_0^i \) denotes the initial condition of the \( i \)-th particle in system \eqref{eq:particle_system}.
     
     Let \( n \geq 1 \) denote the number of alive particles. To each \( n \), we associate a triplet \( \left(\sigma_n, \Gamma_n^N, \xi^{n,\cdot}\right) \), where
    \begin{enumerate}[a)]
        \item \( \sigma_n \) is a stopping time modelling the random time at which $n$ particles out of $N$ have reacted, defined by
        \begin{equation*}
            \sigma_{n} := \inf_{t\in[0,S)}\left\{\sum_{i=1}^N \mathbbm{1}_{(t,\infty)}(\tau^i) = n\right\}.
        \end{equation*}
        Hence,  in $[0,S]$ we get a finite sequence of increasing stopping times $\sigma_0<\dots<\sigma_{n-1}<\sigma_{n}<S$.
        \item \( \Gamma_n^N = \{i_1, \ldots, i_{N-n}\} \) is the random set of indices corresponding to the particles still alive at time \( \sigma_n \). Hence, to the stopping times in a), there is associated a decreasing family of random subsets $\Gamma^N_{n}\subset\Gamma^N_{n-1}\subset \Gamma^N_{0}=\{1,\dots,N\}$. 
        \item \( \xi^{n,\cdot} \) denotes the particle process in $[\sigma_{n-1},S]$, describing the location of the  $N-n+1$ particles alive at time $\sigma_{n-1}$. It is defined as follows.
    \end{enumerate}
 For $j=1,\dots,N-(n-1)$ and for any $t\in [\sigma_{n-1},S]$, the process $\xi^{n,\cdot }_t=\left( \xi^{n,i_1}_t, \dots, \xi_t^{n,i_{N-(n-1)}}\right)$  is the unique strong solution to the system
    \begin{equation}\label{eq:PS_passo_induttivo}
        \begin{aligned}
            &d\xi_t^{n,i_j} = b\left( u_{N,n-1} \left(\cdot,\xi_t^{n,i_j}\right)(t),\nabla u_{N,n-1}\left(\cdot, \xi_t^{n,i_j}\right)(t) \right)dt \,+\, \sqrt{2}dW_t^{i_j};\\
            &\xi_{\sigma_{n-1}}^{n,i_j} = \xi^{i_j}_{\sigma_{n-1}^-},
        \end{aligned}
    \end{equation}
    where the empirical density is computed using only the particles that are alive at the fixed time \( \sigma_{n-1} \), that is,
    \begin{equation*}
        u_{N,n-1}(t,x) := \frac{1}{N} \sum_{i_j \in \Gamma^N_{n-1}} K\left(x - \xi_t^{n,i_j}\right).
    \end{equation*}
    
    The well-posedness of \eqref{eq:PS_passo_induttivo} is assured by the boundedness and Lipschitz continuity of the drift coefficient $b$ from  Proposition \ref{prop:continuity_boundedness_b} and Proposition \ref{prop:lipschitz_prop_b}. Hence, in particular, the process $\xi^{1,\cdot}$ exists and is pathwise unique in $[\sigma_0,S].$

 Finally, the solution $\xi^i_t$ to system \eqref{eq:particle_system}  is constructed, for $i=1,\dots,N$, as 
    \begin{equation}\label{eq:constructive_solution}
        \begin{split}
             \xi_t^i &= \xi_t^{n,i}, \qquad t \in [\sigma_{n-1},\sigma_{n}), \,\, i\in \Gamma_{n-1}^N;\\
             \xi_t^i &\in   \{\Delta\},  \qquad  i\notin \Gamma_{n-1}^N.
         \end{split}
     \end{equation}

    Note that the solution \eqref{eq:constructive_solution} is well-defined and matches continuously at the random partition points $\{\sigma_n\}_n$ of the time interval $[0,S)$, thanks to the specific choice of the initial condition in \eqref{eq:PS_passo_induttivo}. Moreover, given the random time \( \sigma_{n-1} \) defined in \eqref{eq:stopping_time_PS}, and using both \eqref{eq:constructive_solution} and \eqref{eq:PS_passo_induttivo}, the next stopping time \( \sigma_n \) is determined as the first time \( t \geq \sigma_{n-1} \) such that,  for some \( i \in \Gamma^N_{n-1} \),
    \begin{equation*}
        \int_0^{\sigma_{n-1}} \lambda c_0 \exp\left(-\lambda u_N(\cdot, \xi_s^i)(s)\right) ds 
        + \int_{\sigma_{n-1}}^t \lambda c_0 \exp\left(-\lambda u_N(\cdot, \xi_s^{n,i})(s)\right) ds 
        \,\geq\, Z_i.
    \end{equation*}

    This concludes the construction of the solution process to \eqref{eq:particle_system} over $[0,S)$. Uniqueness follows for the uniqueness of the strong solution on each interval $[\sigma_{n-1},\sigma_n)$. In particular, this implies existence of a strong solution \eqref{eq:constructive_solution} which is pathwise unique over the time interval $[0,T]$.
\end{proof}

\section{Conclusions}
This paper provides a rigorous probabilistic foundation for the PDE system \eqref{eq:PDE_rho_regularised},\eqref{eq:def_b_PDE} by identifying and analyzing the underlying microscopic stochastic dynamics. As previously discussed, the aim is to interpret the PDE as the macroscopic mean-field limit of a system of interacting particles driven by Brownian motion. This approach offers a microscale model capable of capturing key aspects of the physical phenomenon, in particular the reactive and non-conservative nature of the dynamics.

A related study was proposed by the same authors in \cite{2024_MTU_arxiv}. However, while both works address the same core problem, they adopt fundamentally different methodologies: the earlier work follows an analytical route, whereas the present one is rooted in a more probabilistic framework. The essential difference lies in the estimation of the law $\nu_t$ of $X_t$  under  survival, using two distinct estimators. Specifically, from \eqref{eq:subprob_killed_SDE_2}, one can estimate the sub-probability measure
$\nu_t(\cdot)=\mathbb P\left( X_t\in \cdot, \Lambda_t<Z\right)$ directly through the empirical particle measure given in \eqref{eq:empirical_measure_mu}. On the other hand, the expected value  representation in \eqref{eq:subprob_killed_SDE_2} leads to the so-called Feynman-Kac-type equation, as it appears in the McKean–Vlasov–Feynman–Kac formulation studied in \cite{2024_MTU_arxiv}. Despite methodological differences, both approaches lead to a representation of the same PDE's solution as the time marginal density of a stochastic process.

A natural further step is to explore particle systems that converge to a PDE model without any regularization by a mollifier $K$, so that  advection and reaction   depend  locally on  $\rho$. That is, we aim to achieve convergence toward the solution of the following equation:
\begin{equation}\label{eq:modello_Natalini}
    \begin{split}
        \partial_t\rho(t,x) \,=\, &\Delta\rho (t,x)\,-\, \nabla \cdot \left[ b\big(\rho (\cdot,x)(t), \nabla \rho(\cdot,x)(t)\big) \rho(t,x)\, \right] -\lambda c_0\left(-\lambda\rho(\cdot,x)(t)\right)\rho(t,x),\\
    \end{split}
\end{equation}
where $(t, x) \in (0, T] \times \mathbb{R}$, with initial conditions $\rho(0,x) = \rho_0(x)$ and $c(0,x) = c_0$, and the advection term $b$ defined in \eqref{eq:def_b_PDE}. The above equation \eqref{eq:modello_Natalini} corresponds to the original model introduced in \cite{2007_AFNT_TPM, 2007_AFNT2_TPM, 2005_GN_NLA}, which initially motivated this line of research. This convergence analysis is carried out in \cite{2025_MTU_arxiv_2}.

Finally, we stress that the main results concerning the model \eqref{eq:PDE_rho_regularised} with the specific advection and reaction terms given in \eqref{eq:def_b_PDE}, namely, Propositions \ref{prop:well_posed_SDE} and \ref{prop:well_posed_PS}, and Theorems \ref{theo:density} and \ref{theo:prob_interpretation} remain valid in a more general setting.
Indeed, since the proofs of these results rely only on the boundedness and Lipschitz continuity properties of the coefficients, the cited results also apply to generic advection and reaction terms satisfying these assumptions.

\section*{Acknowledgments}
The research is carried out within the research project PON 2021(DM 1061, DM 1062) ``Deterministic and stochastic mathematical modelling and data analysis within the study for the indoor and outdoor impact of the climate and environmental changes for the degradation of the Cultural Heritage" of the Università degli Studi di Milano. D.M. and S.U. are members of GNAMPA (Gruppo Nazionale per l’Analisi Matematica, la Probabilità e le loro Applicazioni) of the Italian Istituto Nazionale di Alta Matematica (INdAM).
  

\printbibliography

\end{document}